\documentclass[a4paper, 11pt]{amsart}
\usepackage{amsmath,amsthm,amssymb,amsfonts,color}
\usepackage[arrow,matrix]{xy}
\usepackage{pst-node}
\usepackage{tikz}
\usepackage{etex}

\usepackage{soul}

\usetikzlibrary{arrows,snakes,backgrounds}

\theoremstyle{plain}
\newtheorem{theorem}{Theorem}[section]
\newtheorem{proposition}[theorem]{Proposition}
\newtheorem{lemma}[theorem]{Lemma}
\newtheorem{corollary}[theorem]{Corollary}
\numberwithin{equation}{section}

\theoremstyle{definition}
\newtheorem{definition}[theorem]{Definition}
\newtheorem{remark}[theorem]{Remark}

\newtheorem*{SCRPT}{Strong cohomological rigidity problem for toric varieties}

\newcommand{\C}{\mathbb{C}}

\newcommand{\Q}{\mathbb{Q}}
\newcommand{\R}{\mathbb{R}}
\newcommand{\Z}{\mathbb{Z}}

\newcommand{\CP}{\mathbb{C}P}
\newcommand{\RP}{\mathbb{R}P}

\newcommand{\CC}{\underline{\mathbb{C}}}

\newcommand{\va}{\mathbf{a}}
\newcommand{\vb}{\mathbf{b}}
\newcommand{\vs}{\mathbf{s}}
\newcommand{\vr}{\mathbf{r}}
\newcommand{\vq}{\mathbf{q}}

\newcommand{\vlambda}{\boldsymbol{\lambda}}

\DeclareMathOperator{\Hom}{Hom}
\DeclareMathOperator{\Aut}{Aut}
\DeclareMathOperator{\GL}{GL}

\newcommand{\betti}{\beta}

\begin{document}
\title[Strong Cohomological rigidity of toric varieties]{Strong Cohomological rigidity of toric varieties}
\author[S.Choi]{Suyoung Choi}
\address{Department of Mathematics, Ajou University, San 5, Woncheondong, Yeongtonggu, Suwon 443-749, Korea}
\email{schoi@ajou.ac.kr}

\author[S.Park]{Seonjeong Park}
\address{Division of Mathematical Models, National Institute for Mathematical Sciences, 70 Yuseong-daero 1689beon-gil, Yuseong-gu, Daejeon 305-811, Korea}
\email{seonjeong1124@gmail.com}

\thanks{The first author was supported by Basic Science Research Program through the National Research Foundation of Korea(NRF) funded by the Ministry of Education (NRF-2011-0024975), and the TJ Park Science Fellowship funded by POSCO TJ Park Foundation.
}
\subjclass[2010]{Primary 57S25, 57R19, 14M25}
\keywords{Quasitoric manifold, Complete non-singular toric variety, Strong cohomological rigidity, Toric topology}

\date{\today}
\begin{abstract}
     Every cohomology ring isomorphism between two non-singular complete toric varieties and quasitoric manifolds, respectively, with second Betti number $2$ is realizable by a diffeomorphism and homeomorphism, respectively.
\end{abstract}
\maketitle

\section{Introduction}
    A \emph{toric variety} is a normal algebraic variety of complex dimension $\ell$ with an action of the algebraic torus $(\C^\ast)^\ell$ having an open dense orbit. A typical example of a non-singular complete toric variety is the projective space $\CP^\ell$ of complex dimension $\ell$ with the standard action of $(\C^\ast)^\ell$.

    The \emph{cohomological rigidity problem} for toric varieties poses the question as to whether two non-singular complete toric varieties are diffeomorphic if their cohomology rings are isomorphic as graded rings. Although a cohomology ring is known to be a weak invariant even under homotopy equivalence, no example able to refute the problem has been found yet. On the contrary, many results have been produced in support of the affirmative answers to the problem.
 One of the remarkable results on this topic is that two non-singular complete toric varieties with second Betti number $2$ (or Picard number $2$) are diffeomorphic if and only if their cohomology rings are isomorphic as graded rings, see \cite{ch-ma-su10b}.
 We refer the reader to a survey paper \cite{ch-ma-su11} on this problem.

    On the other hand, it is possible to pose a stronger version of the cohomological rigidity problem for toric varieties as follows.  Throughout this paper, $H^\ast(X)$ denotes the integral cohomology ring of a topological space $X$.

    \begin{SCRPT}
        Let $M$ and $M'$ be non-singular complete toric varieties. If $\varphi$ is a graded ring isomorphism from $H^\ast(M)$ to $H^\ast(M')$, does a diffeomorphism capable of inducing the isomorphism $\varphi$ exist?
    \end{SCRPT}

    The projective space $\CP^1$ is the only non-singular complete toric variety of complex dimension $1$, and it is easy to show that every cohomology ring automorphism is realizable by a diffeomorphism. 
    For the complex dimension $\ell=2$,
    every toric variety of complex dimension $2$ admits a canonical action of the $2$-dimensional compact torus $T^2=(S^1)^2\subset (\C^\ast)^2$.  Orlik and Raymond \cite{OR} showed that real $4$-dimensional compact manifolds which admit well-behaved actions of $T^2$ can be expressed as connected sums of copies of $\CP^2$, $\overline{\CP^2}$ and $\CP^1 \times \CP^1$, and such manifolds are classified by their cohomology rings up to diffeomorphism, where $\overline{\CP^2}$ denotes $\CP^2$ with reversed orientation. According to Wall \cite{Wall64}, each cohomology ring automorphism of such a manifold of dimension $4$ with second Betti number $\betti_2 \leq 10$ is induced by a diffeomorphism. Hence,  one can conclude that the answer to the strong cohomological rigidity problem is affirmative for complex $2$-dimensional non-singular complete toric varieties with $\betti_2 \leq 10$.

    However, the negative answer is also known. For instance, not every cohomology ring automorphism is realizable by a diffeomorphism for a complex $2$-dimensional non-singular complete toric varieties with $\betti_2>10$,
    see \cite{FM}.
    Furthermore, this implies that
    the answer to the strong cohomological rigidity problem for toric varieties of arbitrary dimension with sufficiently large
    $\betti_2$ may be negative.
    Hence, it is reasonable to pose the strong cohomological rigidity problem for toric varieties of arbitrary dimension $\ell$ with small $\betti_2$.
    We note that, because a non-singular complete toric variety with $\betti_2=1$ is the complex projective space $\CP^\ell$, and every automorphism of $H^\ast(\CP^\ell)$ is induced by a diffeomorphism on $\CP^\ell$, the strong cohomological rigidity holds for non-singular complete toric varieties with $\betti_2=1$.

    The aim of the work presented in this paper is to
    study the strong cohomological rigidity problem for non-singular complete toric varieties with $\betti_2=2$.
    We show that the problem can be solved by demonstrating that every cohomology ring automorphism of these toric varieties is realizable by a diffeomorphism.
    Combining our result with the fact that non-singular complete toric varieties with $\betti_2 =2$ are smoothly classified by their cohomology rings \cite{ch-ma-su10b}, we have the following theorem.

    \begin{theorem}\label{thm:submain}
        Every cohomology ring isomorphism between non-singular complete toric varieties with second Betti number $2$ is realizable by a diffeomorphism.
    \end{theorem}

    The notion of a quasitoric manifold was introduced in \cite{DJ} as a topological analogue of a non-singular projective toric variety. A \emph{quasitoric manifold} $M$ is a $2\ell$-dimensional compact smooth manifold with a locally standard $T^\ell$-action whose orbit space can be identified with a simple polytope $P$. Every complex $\ell$-dimensional non-singular projective toric variety with a restricted action of $(\C^\ast)^\ell$ to $T^\ell$ is a quasitoric manifold of dimension $2\ell$. It is noteworthy to remark that every non-singular complete toric variety with $\betti_2=2$ is projective, and hence, it is a quasitoric manifold.
    However, not all quasitoric manifolds can be toric varieties. For example, an equivariant connected sum $\CP^2\#\CP^2$ of two $\CP^2$'s with an appropriate $T^2$-action is a quasitoric manifold with an orbit space a square $\Delta^1 \times \Delta^1$, although it is not a toric variety because it does not admit an almost complex structure. Hence, the class of quasitoric manifolds is larger than that of non-singular projective toric varieties\footnote{In theory, a non-singular non-projective complete toric variety may fail to be a quasitoric manifold.}.

    Note that quasitoric manifolds with $\betti_2 =2$ are topologically classified by their cohomology rings \cite{CPS12}. In this work, we also investigate the strong cohomological rigidity for quasitoric manifolds as follows.
    \begin{theorem}\label{thm:main}
        Every cohomology ring isomorphism between two quasitoric manifolds with second Betti number $2$ is realizable by a homeomorphism.
    \end{theorem}

        The remainder of this paper is organized as follows. In Section~\ref{sec:quasitoric manifolds with second betti 2}, we review the properties of quasitoric manifolds and the topological classification of quasitoric manifolds with $\betti_2=2$. In Section~\ref{sec:wps and connected sum}, we introduce the weighted projective space $\CP^{n+1}_a$ and obtain quasitoric manifolds over $\Delta^n\times\Delta^1$ by carrying out an equivariant connected sum $\CP^{n+1}_a\#\CP^{n+1}_a$ or $\CP^{n+1}_a\#\overline{\CP^{n+1}_a}$. By using this, we show that every cohomology ring automorphism of such a quasitoric manifold is realizable by a diffeomorphism. Section~\ref{sec:two-stage GB} discusses the realizability of a cohomology ring automorphism for a non-singular complete toric variety with $\betti_2=2$. In Section~\ref{sec:quasitoric over n x m}, we consider quasitoric manifolds over the product of simplices $\Delta^n\times \Delta^m$ which are not non-singular complete toric varieties. Finally, we complete the proofs of Theorems~\ref{thm:submain} and~\ref{thm:main} in Section~\ref{sec:proof of main}.

\section{Quasitoric manifolds with second Betti number 2}\label{sec:quasitoric manifolds with second betti 2}

    In this section, we first review the general properties of quasitoric manifolds from \cite{DJ}, \cite{BP} and \cite{ch-ma-su10a}.
    We partially focus on the case for which the second Betti number is $2$. In addition, we recall the classification results in \cite{ch-ma-su10b} and \cite{CPS12}.

    Let $M$ be a $2\ell$-dimensional quasitoric manifold over an $\ell$-dimensional simple polytope $P$ with $d$ facets (codimension-$1$ faces). Let $F$ be a $k$-dimensional face of $P$. Note that for the orbit map $\rho\colon M\to P$ and for a point $x\in\rho^{-1}(F^\circ)$, the isotropy subgroup at $x$ is independent of the choice of $x$ and is a codimension-$k$ subtorus of $T^\ell$, where $F^\circ$ denotes the interior of $F$.
    If $F$ is a facet of $P$, then $\rho^{-1}(F)$ is fixed by a circle subgroup of $T^\ell$.
    We define a function $\vlambda \colon \{F_1, \ldots, F_d\} \to \Hom(S^1 , T^\ell) \cong \Z^\ell$, known as the \emph{characteristic function of $M$}, such that $\vlambda(F_i)$ fixes the \emph{characteristic submanifold} $M_i:=\rho ^{-1}(F_i)$ for $i=1, \ldots, d$, where $\{F_1, \ldots, F_d\}$ is the set of facets of $P$. We note that $\vlambda$ satisfies the following \emph{non-singularity condition};
    \begin{equation} \label{eq:non-singularity condition}
        \begin{split}
            &\vlambda(F_{i_1}),\ldots,\vlambda(F_{i_\alpha}) \text{ form a part of an integral basis of }\Z^\ell\\
            &\text{whenever the intersection }F_{i_1}\cap\cdots\cap F_{i_\alpha}\text{ is non-empty}.
        \end{split}
    \end{equation}

    Conversely, let us consider a function $\vlambda \colon \{F_1, \ldots, F_d\} \to \Z^\ell$ satisfying \eqref{eq:non-singularity condition} and its matrix representation $\Lambda = \begin{pmatrix} \vlambda(F_1)&\cdots &\vlambda(F_d) \end{pmatrix}$, called a \emph{characteristic matrix}. For a characteristic matrix $\Lambda$ and a face $F$ of $P$, we denote by $T(F)$ the subgroup of $T^\ell$ corresponding to the unimodular subspace of $\Z^\ell$ spanned by $\vlambda(F_{i_1}),\ldots,\vlambda(F_{i_\alpha})$, where $F=F_{i_1}\cap\cdots\cap F_{i_\alpha}$.
    For each point $p\in P$, let  $F(p)$ denote the face of $P$ containing $p \in P$ in its relative interior.
    Then, we construct a manifold
    \begin{equation}\label{eq:construction of M(P,L)}
        M(P,\Lambda):=T^\ell\times P/\sim,
    \end{equation}
    where $(t,p)\sim (s,q)$ if and only if $p=q$ and $t^{-1}s\in T(F(p))$.
    Then, the standard $T^\ell$-action on $T^\ell$ induces a locally standard $T^\ell$-action on $M(P,\Lambda)$, and $M(P,\Lambda)$ is indeed a quasitoric manifold over $P$ whose characteristic function is $\vlambda$.
    Note that the two vectors $\vlambda(F_i)$ and $-\vlambda(F_i)$ determine the same circle subgroup of $T^n$. Hence, if $\Lambda'$ is a matrix obtained from $\Lambda$ by changing the signs of some columns,
    then $M(P,\Lambda')$ is equal to $M(P,\Lambda)$.

    Set $\mathfrak{F}(P)=\{F_1,\ldots,F_d\}$ and define a map $\Theta\colon\mathfrak{F}(P)\to\Z^d$ by $\Theta(F_i)=\mathbf{e}_i$, where $\mathbf{e}_i$ is the $i$th standard basis vector. Using $\Theta$, we can construct a $T^d$-manifold $\mathcal{Z}_P=T^d\times P/\sim$ as in~\eqref{eq:construction of M(P,L)}. Then the dimension of $\mathcal{Z}_P$ is equal to $d+\ell$. The $T^d$-manifold $\mathcal{Z}_P$ is referred to as a \emph{moment-angle manifold} of $P$. For instance, $\mathcal{Z}_{\Delta^\ell}$ is the $(2\ell+1)$-dimensional sphere $S^{2\ell+1}$. Note that for two simple polytopes $P$ and $Q$, we have $\mathcal{Z}_{P \times Q} = \mathcal{Z}_P \times \mathcal{Z}_Q$. Hence, $\mathcal{Z}_{\Delta^n \times \Delta^m} = S^{2n+1} \times S^{2m+1}$.
    Let us consider the map $\Z^d\to\Z^\ell$ which makes the following diagram commute
        \begin{equation*}
        \xymatrix{
         \Z^d\ar[rr]&&\Z^\ell\\
         & \mathfrak{F}(P)\ar[ul]^{\Theta}\ar[ur]_{\vlambda}&
        }.
    \end{equation*}
    Then this map can be regarded as a homomorphism defined by $\mathbf{x}\mapsto \Lambda\mathbf{x}$ for every $\mathbf{x}\in\Z^d$. Henceforth, this homomorphism is denoted by $\vlambda$ unless this is confusing.
    Let $K$ be the subtorus of $T^d$ corresponding to $\ker\vlambda$.
    Then $K$ acts freely on $\mathcal{Z}_P$, and the orbit space of $K$ on $\mathcal{Z}_P$ is the quasitoric manifold $M(P,\Lambda)$.

    Two quasitoric manifolds $M$ and $M'$ over $P$ are said to be \emph{equivalent} if there is a $\theta$-equivariant homeomorphism $f\colon M\to M'$, i.e., $f(t\cdot x)=\theta(t)\cdot f(x)$ for $t\in T^\ell$ and $x\in M$, which covers the identity map on $P$ for some automorphism $\theta$ of $T^\ell$. Thus, $M(P,\Lambda)$ and $M(P,\Lambda')$ are equivalent if there is an element $G$ in the general linear group $\GL(\ell,\Z)$ of rank $\ell$ over $\Z$ such that $\Lambda' = G \Lambda$.

    There is a well-known formula for the cohomology ring of a quasitoric manifold with $\Z$-coefficients. Let $M$ be a quasitoric manifold over $P$ with the characteristic matrix $\Lambda=(\lambda_{ij})_{1\leq i\leq \ell\atop 1\leq j\leq d}$. Then,
    \begin{equation}\label{eq:cohomology of quasitoric manifold}
        H^\ast(M(P,\Lambda))=\Z[x_1,\ldots,x_d]/\mathcal{I}_P+\mathcal{J},
    \end{equation}
    where $x_i$ is the degree-two cohomology class dual to the characteristic submanifold $M_i$, $\mathcal{I}_P$ is the homogeneous ideal generated by all square-free monomials $x_{i_1}\cdots x_{i_\alpha}$ such that $F_{i_1}\cap\cdots \cap F_{i_\alpha}$ is empty, and $\mathcal{J}$ is the ideal generated by linear forms $\lambda_{i1}x_1+\cdots+\lambda_{id}x_d$, $1\leq i\leq \ell$. Note that the second Betti number of $M$ is equal to $d-\ell$.

    Let $M$ be a quasitoric manifold with second Betti number $\beta_2=2$. Then the orbit space of $M$ is a polytope of dimension $\ell$ with $\ell+2$ facets. Hence, the orbit space is a product of two simplices $\Delta^n\times \Delta^m$ (see \cite{Gru}) for some $n$ and $m$ satisfying $n+m = \ell$. Let $\{F_{1}, \ldots, F_{n+1}\}$ and $\{F'_1, \ldots, F'_{m+1}\}$ be the sets of facets of $\Delta^n$ and $\Delta^m$, respectively. Then, each facet
    of $\Delta^n \times \Delta^m$ is either of the form $F_i \times \Delta^m$ or $\Delta^n \times F'_j$. We may assign an order to the facets of $\Delta^n\times \Delta^m$ by
    $$F_1 \times \Delta^m , \Delta^n\times F_1', F_2\times \Delta^m, \ldots, F_{n+1} \times \Delta^m, \Delta^n \times F'_2, \ldots, \Delta^n \times F'_{m+1}.$$
    Since the last  $\ell$ facets meet at a vertex, up to equivalence, we may assume that the last $\ell$ columns of the characteristic matrix $\Lambda$ corresponding to $M$ form an identity matrix. Furthermore, by the non-singularity condition \eqref{eq:non-singularity condition}, it becomes clear that
    \begin{equation}\label{eq:char mx of M}
        \Lambda=
        \begin{pmatrix}
            -1&-b_1&1& & & & & \\
            \vdots &\vdots&&\ddots & & &0& \\
             -1&-b_n&& &1& & & \\
             -a_1&-1&& & &1& & \\
             \vdots &\vdots&&0& & &\ddots& \\
             -a_m&-1&& & & & &1
        \end{pmatrix},
    \end{equation} where $1-a_jb_i=\pm 1$ for $i=1,\ldots,n$ and $j=1,\ldots,m$. See~\cite{CPS12} for more details.
 From now on, $M_{\mathbf{a},\mathbf{b}}$ denotes the quasitoric manifold $M(\Delta^n\times\Delta^m,\Lambda)$  for $\Lambda$ in~\eqref{eq:char mx of M}, where $\mathbf{a}=(a_1,\ldots,a_m)$ and $\mathbf{b}=(b_1,\ldots,b_n)$.
    By \eqref{eq:cohomology of quasitoric manifold}, the cohomology ring of $M_{\va,\vb}$ with $\Z$-coefficients is
    \begin{equation}\label{eq:cohomology of M}
        H^\ast(M_{\va,\vb})=\Z[x_1,x_2]\left/\left\langle x_1\prod_{i=1}^n(x_1+b_ix_2),~x_2\prod_{j=1}^m(a_jx_1+x_2)\right\rangle\right..
    \end{equation}

    A \emph{generalized Bott tower} of height $h$, or an \emph{$h$-stage generalized Bott tower}, is a sequence
    \begin{equation*}
        B_h \stackrel{\pi_h}\longrightarrow B_{h-1}\stackrel{\pi_{h-1}}\longrightarrow\cdots \stackrel{\pi_2}\longrightarrow B_1\stackrel{\pi_1}\longrightarrow B_0=\{\mbox{a point}\}
    \end{equation*} of manifolds $B_i=P(\CC\oplus\bigoplus_{j=1}^{n_i}\xi_{i,j})$, where $\CC$ is the trivial line bundle, $\xi_{i,j}$ is a complex line bundle over $B_{i-1}$ for each $i=1,\ldots,h$, and $P(\cdot)$ stands for the projectivization. We refer to $B_i$ as an \emph{$i$-stage generalized Bott manifold}. We remark that a $2$-stage generalized Bott manifold provided by $n=m=1$ is known as a \emph{Hirzebruch surface} \cite{DJ}. Note that $h$-stage generalized Bott manifolds are non-singular projective toric varieties with $\beta_2=h$, and are quasitoric manifolds over a product of $h$ simplices. Moreover, by \cite{ch-ma-su10a}, a quasitoric manifold over a product of simplices has a
    non-singular complete toric variety structure if and only if it is equivalent to a generalized Bott manifold. Hence, every non-singular complete toric variety
    with $\beta_2=2$ is a two-stage generalized Bott manifold.

    For simplicity, for every complex line bundle $L$ over a base $B$, the $a$-times tensor bundle of $L$ is denoted by $L^{a}$.
    If $\vb=\mathbf{0}$, then $M_{\va,\mathbf{0}}$ is equivalent to a two-stage generalized Bott manifold $P(\CC\oplus\bigoplus_{j=1}^m \gamma^{a_j})$, where $\gamma$ is a tautological line bundle over $\CP^n$. Furthermore, in \eqref{eq:cohomology of M}, the generator $x_1$ of $H^\ast(M_{\va,\mathbf{0}})$ is $-c_1(\gamma)$, the negative of the first Chern class of $\gamma$, and the generator $x_2$ of $H^\ast(M_{\va,\mathbf{0}})$ is the negative of the first Chern class of the tautological line bundle over $P(\CC\oplus\bigoplus_{j=1}^m \gamma^{a_j})$.
    On the other hand, if $\va=\mathbf{0}$, then a quasitoric manifold $M_{\mathbf{0},\vb}$ is equivalent to a two-stage generalized Bott manifold $P(\CC\oplus\bigoplus_{i=1}^n\eta^{b_i})$, where $\eta$ is a tautological line bundle over $\CP^m$, see \cite{ch-ma-su10a}. Similarly, in \eqref{eq:cohomology of M}, the generator $x_2$ of $H^\ast(M_{\mathbf{0},\vb})$ is $-c_1(\eta)$ and the generator $x_1$ of $H^\ast(M_{\mathbf{0},\vb})$ is the negative of the first Chern class of the tautological line bundle over $P(\CC\oplus \bigoplus_{i=1}^n\eta^{b_i})$.

    The following theorem provides a smooth classification of two-stage generalized Bott manifolds.
    \begin{theorem}\cite{ch-ma-su10b}\label{thm:CR of GB}
        Let $B_2:=P(\CC\oplus\bigoplus_{j=1}^m\gamma^{a_j})$ and $B_2':=P(\CC\oplus\bigoplus_{j=1}^m\gamma^{a_j'})$, where $\gamma$ denotes the tautological line bundle over $B_1=\CP^n$. The following are equivalent.
        \begin{enumerate}
            \item There exist $\epsilon=\pm1$ and $w\in\Z$ such that $$(1+\epsilon wx_1)\prod_{j=1}^m(1+\epsilon(a_j'+w)x_1)=\prod_{j=1}^m(1+a_jx_1)\in H^\ast(B_1),$$
                where $x_1=-c_1(\gamma)\in H^2(B_1)$.
            \item Two generalized Bott manifolds $B_2$ and $B_2'$ are diffeomorphic.
            \item The cohomology rings $H^\ast(B_2)$ and $H^\ast(B_2')$ are isomorphic as graded rings.
        \end{enumerate}
    \end{theorem}

    If neither $\va$ nor $\vb$ is a zero vector, then $M_{\va,\vb}$ cannot be equivalent to a two-stage generalized Bott manifold. Moreover, from the non-singularity condition of \eqref{eq:char mx of M}, either the nonzero entries of $\va$ are $\pm2$ and the nonzero entries of $\vb$ are $\pm1$, or the nonzero entries of $\va$ are $\pm1$ and the nonzero entries of $\vb$ are $\pm2$.

    The following theorem gives a topological classification of quasitoric manifolds with $\betti_2=2$.
    \begin{theorem}\cite{CPS12}\label{thm:main of pjm}
        Two quasitoric manifolds with second Betti number $2$ are homeomorphic if and only if their integral cohomology rings are isomorphic as graded rings.

        Furthermore, a quasitoric manifold $M$ with $\betti_2=2$ which is not equivalent to a
        generalized Bott manifold is homeomorphic to $M_{\vs,\vr}$ for some nonzero vectors $$\vs:=(\underbrace{2,\ldots,2}_s,0,\ldots,0)\in\Z^m \mbox{ and } \vr:=(\underbrace{1,\ldots,1},0,\ldots,0)\in\Z^n,$$ where $s\leq\lfloor \frac{m+1}{2}\rfloor$ and $r\leq\lfloor \frac{n+1}{2}\rfloor$.
        In particular, in the case where $n>1$ and $m>1$, all $M_{\vs,\vr}$'s are distinct and they cannot be homeomorphic to generalized Bott manifolds. In other cases, $M_{\vs,\vr}$ is homeomorphic to
        \begin{enumerate}
            \item $M_{\mathbf{0},1}=\CP^{m+1}\#\overline{\CP^{m+1}}
            $ if $n=1$ and $m$ is even;
            \item either $M_{\mathbf{0},1}$ or $M_{(2,0,\ldots,0),1}=
            \CP^{m+1}\#\CP^{m+1}$ if $n=1$ and $m$ is odd;
            \item $M_{2,\mathbf{0}}$ if $n$ is even and $m=1$; and
            \item either $M_{2,\mathbf{0}}$ or $M_{2,(1,0,\ldots,0)}$ if $n$ is odd and $m=1$,
        \end{enumerate} where $\#$ denotes an equivariant connected sum and $\overline{\CP^{m+1}}$ denotes $\CP^{m+1}$ with reversed orientation.
    \end{theorem}

    Since the orbit space of a quasitoric manifold $M_{\va,\vb}$ is $\Delta^n \times \Delta^m$,
    $M_{\va,\vb}$ is a quotient of $\mathcal{Z}_{\Delta^n\times\Delta^m}=S^{2n+1}\times S^{2m+1}$ by an action of $T^2$.
    More precisely, let
    us define a free action of two-torus $K_{\va,\vb}$ on $S^{2n+1}\times S^{2m+1}$ by
    \begin{equation*}
        \begin{split}
            &(t_1,t_2)\cdot((w_1,\ldots,w_{n+1}),(z_1,\ldots,z_{m+1}))\\
            &\quad=((t_1t_2^{b_1}w_1,\ldots,t_1t_2^{b_n}w_n,t_1w_{n+1}), (t_1^{a_1}t_2z_1,\ldots,t_1^{a_m}t_2z_m,t_2z_{m+1})).
        \end{split}
    \end{equation*}
    Then the orbit space $S^{2n+1}\times S^{2m+1}/K_{\va,\vb}$ is the quasitoric manifold $M_{\va,\vb}$.

    \begin{remark} \label{rem:Aut_in_GL}
        Let $\varphi$ be a graded ring automorphism of $H^\ast(M_{\va,\vb})$. Then there is a matrix $(g_{ij})_{i,j=1,2}$ such that $g_{11}g_{22}-g_{12}g_{21}=\pm1$ and
        $$\begin{pmatrix}\varphi(x_1)\\\varphi(x_2)\end{pmatrix} = \begin{pmatrix}g_{11}&g_{12}\\g_{21}&g_{22}\end{pmatrix} \begin{pmatrix}x_1\\x_2\end{pmatrix}.$$ Hence, $\Aut(H^\ast(M_{\va,\vb}))$ can be regarded as a subgroup of $\GL(2,\Z)$.
    \end{remark}

\section{Weighted projective spaces and their connected sum}\label{sec:wps and connected sum}

    It is well-known that the quasitoric manifold $M_{1,\mathbf{0}}$ over $\Delta^n\times\Delta^1$ is the connected sum $\CP^{n+1}\#\overline{\CP^{n+1}}$, and the quasitoric manifold $M_{1,(2,0,\ldots,0)}$ is the connected sum $\CP^{n+1}\#\CP^{n+1}$. In this section, we show that quasitoric manifolds $M_{2,\mathbf{0}}$ and $M_{2,(1,0,\ldots,0)}$ over $\Delta^n\times\Delta^1$ can be expressed as equivariant connected sums of weighted projective spaces, before considering the realizability of the automorphism of $H^\ast(M_{a,\vb})$ when $a=1$ or $a=2$.

    Let us first consider the definitions and properties of weighted projective spaces.

    \begin{definition}\label{def:weighted complex projective space}
        Let $\vq=(q_0,\ldots,q_{\ell})$ be an $(\ell+1)$-tuple of positive integers, with $\gcd(q_0,\ldots,q_{\ell})=1$. The (complex) \emph{weighted projective space} of weight $\vq$, denoted by $\CP^{\ell}_\vq$, is defined as the quotient of
        $\C^{\ell+1}\setminus\{0\}$ by the weighted action of $\C^\ast$, $$\zeta\cdot(z_0,\ldots,z_{\ell}) \mapsto (\zeta^{q_0}z_0,\ldots,\zeta^{q_{\ell}}z_{\ell}).$$
        Alternatively, $\CP^\ell_\vq$ can be realized as the quotient of the unit sphere $S^{2\ell+1}\subset\C^{\ell+1}$ by the action of $S^1$, which is obtained by the restriction of the above action of $\C^\ast$ to the unit circle $S^1$.
    \end{definition}
    Note that
    if $q_0=\cdots=q_{\ell}=1$, then $\CP^\ell_\vq$ is the ordinary projective space $\CP^\ell$. The image of $(\C^\ast)^{\ell+1}\subset\C^{\ell+1}\setminus\{0\}$ in $\CP^\ell_\vq$ is the quotient $(\C^\ast)^{\ell+1}/\C^\ast$, where we regard $\C^\ast$ as the subgroup of $(\C^\ast)^{\ell+1}$ via the map $\zeta\mapsto(\zeta^{q_0},\ldots,\zeta^{q_{\ell}})$. Then, the action of $(\C^\ast)^{\ell+1}$ on $\C^{\ell+1}\setminus\{0\}$ descends to an action of $(\C^\ast)^\ell\cong(\C^\ast)^{\ell+1}/\C^\ast$ on $\CP^\ell_\vq$. Furthermore, $\CP^\ell_\vq$ is a projective toric variety which is not necessarily non-singular.

    Note that $\CP^\ell_\vq$ is equipped with an action of the $\ell$-dimensional torus $T_\vq^\ell=(S^1)^{\ell+1}/j_\vq(S^1)$, where $j_\vq\colon S^1\to (S^1)^{\ell+1}$ is the embedding defined by $j_\vq(\zeta)=(\zeta^{q_0},\ldots,\zeta^{q_{\ell}}).$ It is well-known that $\CP^\ell_\vq$ with the action of $T^\ell_\vq$ is a toric K\"{a}hler orbifold, see \cite{Gau} for more details.

    We can also consider real weighted projective spaces as follows.

    \begin{definition}
        Let $\vq:=(q_0,\ldots,q_\ell)$ be an $(\ell+1)$-tuple of integers. The \emph{real weighted projective space} $\RP^\ell_\vq$ is the quotient of the unit sphere $S^{\ell}\subset \R^{\ell+1}$ by the action of $\Z_2=\{\pm1\}$ defined by
        $$(-1)\cdot(x_0,\ldots,x_\ell)=((-1)^{q_0}x_0,\ldots,(-1)^{q_\ell}x_\ell).$$
        Hence, if all $q_i$'s are odd, then $\RP^\ell_\vq$ is the ordinary real projective space $\RP^\ell$.
    \end{definition}

    Note that the real projective space $\RP^\ell_\vq$ is the fixed set of the conjugation action on the weighted projective space $\CP^\ell_\vq$.

    As mentioned in the introduction, a quasitoric manifold is a topological generalization of a non-singular projective toric variety.
    The notion of a projective toric variety, which is not necessarily non-singular, is also topologically generalized to that of a \emph{quasitoric orbifold}. This generalization was introduced by several authors such as \cite{DJ}, \cite{HM}, and \cite{Po-Sa}.

    Suppose that $P$ is a simple polytope of dimension $\ell$ with $d$ facets $F_1,\ldots,F_d$. A function $\vlambda\colon\{F_1,\ldots,F_d\}\to\Z^\ell$ is called a \emph{rational characteristic function} if it satisfies
    \begin{equation*}
        \begin{split}
            &\vlambda(F_{i_1}),\ldots,\vlambda(F_{i_\alpha}) \text{ are linearly independent over }\Z\text{ whenever}\\
            &\text{the intersection }F_{i_1}\cap\cdots\cap F_{i_\alpha}\text{ is non-empty}.
        \end{split}
    \end{equation*} Each vector $\vlambda(F_i)$ is the \emph{rational characteristic vector} corresponding to $F_i$. Let $K$ be the subtorus of $T^d$ corresponding to the kernel of $\vlambda$. Then $K$ acts on $\mathcal{Z}_P$ with finite isotropy groups. We denote by $Q(P,\vlambda)$ the orbit space of $K$ on $\mathcal{Z}_P$  and call it the \emph{quasitoric orbifold} corresponding to $(P,\vlambda)$. If we assign an order to the set of facets of $P$, the rational characteristic function $\vlambda$ can be represented by the \emph{rational characteristic matrix} $\Lambda=\begin{pmatrix}
        \vlambda(F_1) &\cdots&\vlambda(F_d)
    \end{pmatrix}.$
    For simplicity,
    we use the notation $Q(P,\Lambda)$ instead of $Q(P,\vlambda)$ provided this does not cause confusion.

    Note that $\mathcal{Z}_{\Delta^\ell}$
    is
    $S^{2\ell+1}$ and the subtorus $K$ corresponding to the kernel of a rational characteristic function on $\Delta^\ell$ is a circle with a suitable weight. Hence, the weighted projective space $\CP^\ell_\vq$ is a quasitoric orbifold over $\Delta^\ell$.

    In particular, for a positive integer $a$, let $\vq:=(1,\ldots,1,a)\in\Z^{n+2}$. We specify $\CP^{n+1}_a:=\CP^{n+1}_\vq$ and $T^{n+1}_a:=T^{n+1}_\vq$. That is, $\CP^{n+1}_a$ is the quotient of $S^{2n+3}$ by the action of $S^1$ with the weight $(1,\ldots,1,a)$, $$\zeta\cdot (z_0,\ldots,z_{n+1})\mapsto (\zeta z_0,\ldots,\zeta z_{n},\zeta^a z_{n+1}).$$
    For each $\mathbf{z}=(z_0,\ldots,z_{n+1})$
    in $S^{2n+3}$, the isotropy group of the action of $S^1$ at $\mathbf{z}$ is the identity except if
    $\mathbf{z}=(0,\ldots,0,1)$.
    The isotropy group at $\mathbf{z}=(0,\ldots,0,1)$ is $\mu_a$, the group of the $a$th roots of $1$.
    Therefore, the
    weighted projective space $\CP_{a}^{n+1}$ has a unique singularity at the point $[0,\ldots,0,1]$, modeled on $\C^{n+1}/\mu_a$.
  Let us find the rational characteristic function $\vlambda$ corresponding to $\CP^{n+1}_a$. Note that for each $i=0,\ldots,n+1$, the sub-orbifold $Q_i$ of $\CP^{n+1}_a$ described by $z_i=0$ is fixed by the quotient of the $(i+1)$th coordinate circle of $T^{n+2}$.
    We identify $T^{n+1}_a$ and $T^{n+1}$ via the map which sends the $(i+1)$th coordinate circle of $T^{n+2}$ to the $i$th coordinate circle of $T^{n+1}$ for $i=1, \ldots, n+1$. Since $[\zeta,1,\ldots,1] = [1,\zeta^{-1},\ldots,\zeta^{-1},\zeta^{-a}]$ in $T^{n+1}_a$, the first coordinate circle of $T^{n+2}$ is identified with the circle subgroup of $T^{n+1}$ generated by $(\zeta^{-1},\ldots,\zeta^{-1},\zeta^{-a})$. Moreover, the torus $T^{n+1}$ acts on $\CP^{n+1}_a$ as follows:
    $$(t_1,\ldots,t_{n+1})\cdot[z_0,\ldots,z_{n+1}]=[z_0,t_1z_1,\ldots,t_{n+1}z_{n+1}].$$
    Then, for each $i=1,\ldots,n+1$, the sub-orbifold $Q_i$ is fixed by the $i$th coordinate circle of $T^{n+1}$, and $Q_0$ is fixed by the circle generated by $(-1,\ldots,-1,-a)$ in $\Z^{n+1} =\Hom(S^1, T^{n+1})$.
    Let us denote  by $F_i$ the facet of $\Delta^{n+1}$ corresponding to $Q_i$. Then $\vlambda(F_0)=(-1,\ldots,-1,-a)$ and $\vlambda(F_i)=\mathbf{e}_i$ for $i=1,\ldots, n$. 
    Hence, the rational characteristic matrix corresponding to $\CP^{n+1}_a$ is
    \begin{equation}\label{eq:canonical char mx for CP_a}
        \Lambda_a:= \begin{pmatrix} \vlambda(F_0) & \vlambda(F_1) & \cdots  & \vlambda(F_{n+1}) \end{pmatrix} =
        \begin{pmatrix}
            -1&1&&    &      &    \\
            -1    &    &1&      &    &    \\
            \vdots    &    &&\ddots&    &    \\
            -1    &    &&      &1&    \\
            -a    &    &&      &    &1
        \end{pmatrix}.
    \end{equation}
    In particular, a fan of $\CP_{a}^{n+1}$ as a projective toric variety is obtained by taking the cones generated by all proper subsets of
    $$\{ -\mathbf{e}_1-\cdots-\mathbf{e}_n-a\mathbf{e}_{n+1}, \mathbf{e}_1,\ldots,\mathbf{e}_{n+1}\}.$$

   On the other hand, consider $(n+1)\times(n+2)$ matrices of the form
        \begin{equation*}
            \Lambda=\begin{pmatrix}
                \pm1&\pm1&&    &      &    \\
                \pm1&    &\pm1&      &    &    \\
                \vdots&    &&\ddots&    &    \\
                \pm1&    &&      &\pm1&    \\
                \pm a&    &&      &    &\pm1
            \end{pmatrix}.
        \end{equation*}
    Then $\Lambda$ is a rational characteristic matrix on $\Delta^{n+1}$. Because a row operation of $\Lambda$ whose determinant is $\pm1$ corresponds to an automorphism of $T^{n+1}$, and changing the signs of column vectors does not affect the subgroup generated by these column vectors, it is clear that $Q(\Delta^{n+1},\Lambda)$ is equivalent to the weighted projective space $\CP^{n+1}_a$ with a suitable action of $T^{n+1}$.

Now, let us consider a smooth manifold  $\CP_a^{n+1}\#\overline{\CP_{a}^{n+1}}$ obtained by the (equivariant) connected sum of $\CP^{n+1}_a$ and $\overline{\CP^{n+1}_a}$ at their singular points.  More precisely,
   let $D$ and $D'$ be
   closed balls in $\CP_a^{n+1}$ and $\overline{\CP^{n+1}_a}$ containing the singular
    point, respectively, which is a sub-orbifold with a boundary diffeomorphic
    to $D^{2(n+1)}/\mu_a$, where $D^{2(n+1)}$ is the closed unit ball in $\C^{n+1}$.
    By deleting the interiors of the balls $D$ and $D'$ in $\CP_{a}^{n+1}$ and $\overline{\CP_{a}^{n+1}}$, respectively, and attaching the resulting punctured manifolds $\CP_{a}^{n+1}\setminus D^\circ$ and $\overline{\CP_{a}^{n+1}}\setminus {D^\prime}^\circ$ to each other by a diffeomorphism $\partial D\cong S^{2n+1}/\mu_a\cong \partial D'$,
    we
    obtain a smooth manifold $\CP_a^{n+1}\#\overline{\CP_{a}^{n+1}}$\footnote{Note that if we do a connected sum at a non-singular point, then the connected sum $\CP^{n+1}_a\#\overline{\CP^{n+1}_a}$ or $\CP^{n+1}_a\#\CP^{n+1}_a$ has still singular points. In this paper, we consider only the case that both $\CP^{n+1}_a\#\overline{\CP^{n+1}_a}$ and $\CP^{n+1}_a\#\CP^{n+1}_a$ are smooth manifolds which are the connected sums at singular points.}. This can be described in terms of toric topological language. Let us review the construction of the equivariant connected sum for quasitoric orbifolds $Q(P',\vlambda')$ and $Q(P'',\vlambda'')$. Let $v'=F_1'\cap\cdots\cap F_\ell'$ and $v''=F_1''\cap\cdots\cap F_\ell''$ be the vertices of $P'$ and $P''$, respectively, such that $\vlambda'(F_i')=\vlambda''(F_i'')$ for $i=1,\ldots,\ell$. Then the \emph{connected sum} of $P'$ and $P''$ with respect to $v'$ and $v''$ is combinatorially equivalent to the polytope formed by deleting the small balls of $v$ and $v'$ of $P$ and $P'$, respectively, and gluing together the resulting neighborhoods. The hyperplanes containg $F_i'$ and $F_i''$, respectively, must be attached to another  for $i=1,\ldots,\ell$. When the choices of $v'$ and $v''$ with the order of facets are clear, the connected sum is denoted by $P'\# P''$. Indeed, $P' \# P''$ is combinatorially equivalent to a simple polytope due to \cite{BR}. The \emph{equivariant connected sum} of $Q(P',\vlambda')$ and $Q(P'',\vlambda'')$ is the quasitoric orbifold corresponding $Q(P' \# P'' , \vlambda)$ where $\vlambda$ is a characteristic function naturally defined by $\vlambda'$ and $\vlambda''$. Indeed, $\CP_a^{n+1}\#\overline{\CP_{a}^{n+1}}$ is the equivaraint connected sum of two quasitoric orbifolds.  Near the singular points both $[0,\ldots,0,1]\in\CP_{a}^{n+1}$ and $[0,\ldots,0,1]\in \overline{\CP_{a}^{n+1}}$, they have the same singularity and the same characteristic vectors: $\mathbf{e}_1,\ldots,\mathbf{e}_n$, and $-\mathbf{e}_1-\cdots-\mathbf{e}_n-a\mathbf{e}_{n+1}$. Hence, we can carry out an equivariant connected sum $\CP^{n+1}_a\#\overline{\CP^{n+1}_a}$ by removing the singular points  as in Figure~\ref{fig:connected sum and hirzebruch}.

        \begin{figure}[h]
            \begin{center}
                \begin{tikzpicture}[scale=.8]
                    \draw (1.5,1)--(3,2)--(1.5,3)--cycle;
                    \draw (5,2)--(6.5,1)--(6.5,3)--cycle;
                    \draw (9.5,1)--(11.5,1)--(11.5,3)--(9.5,3)--cycle;
                    \draw (4,2) node{$\#$};
                    \draw (8,2) node{$=$};
                    \draw (2,2) node{$\circlearrowleft$};
                    \draw (6,2) node{$\circlearrowleft$};
                    \draw (10.5,2) node{$\circlearrowleft$};
                    \draw (2.5,3) node{\tiny{$\begin{pmatrix}-1\\-a\end{pmatrix}$}};
                    \draw (2.5,1.1) node{\tiny{$\begin{pmatrix}1\\0\end{pmatrix}$}};
                    \draw (1.1,2) node{\tiny{$\begin{pmatrix}0\\1\end{pmatrix}$}};
                    \draw (5.5,3) node{\tiny{$\begin{pmatrix}-1\\-a\end{pmatrix}$}};
                    \draw (5.5,1.1) node{\tiny{$\begin{pmatrix}1\\0\end{pmatrix}$}};
                    \draw (7,2) node{\tiny{$\begin{pmatrix}0\\-1\end{pmatrix}$}};
                    \draw (10.5,3.5) node{\tiny{$\begin{pmatrix}-1\\-a\end{pmatrix}$}};
                    \draw (10.5,0.5) node{\tiny{$\begin{pmatrix}1\\0\end{pmatrix}$}};
                    \draw (9.1,2) node{\tiny{$\begin{pmatrix}0\\1\end{pmatrix}$}};
                    \draw (12,2) node{\tiny{$\begin{pmatrix}0\\-1\end{pmatrix}$}};
                    \fill (3,2) circle(2pt);
                    \fill (5,2) circle(2pt);
                \end{tikzpicture}
            \end{center}
            \caption{Illustration of $\CP_a^2\#\overline{\CP_a^2}$} \label{fig:connected sum and hirzebruch}
        \end{figure}
        Note that the orientation of $\CP^{n+1}_a$ is associated with the orientation of $\Delta^{n+1}\subset\R^{n+1}$ and the columns of the characteristic matrix are determined up to sign. Then the characteristic matrix of $\CP^{n+1}_a\#\overline{\CP^{n+1}_a}$ is
        \begin{equation}\label{eq:char mx of connected sum}
            \begin{pmatrix}
                -1&&1&&    &      &    \\
                   -1&& &1&      &    &    \\
                    \vdots&&&&\ddots&    &    \\
                    -1&&&&      &1&    \\
                     -a&-1&&&      &    &1
            \end{pmatrix}.
        \end{equation}

    Now, consider a two-stage generalized Bott manifold $P(\CC\oplus\gamma^{ a})$, where $\gamma$ is the tautological line bundle over $\CP^n$. The relationship between the weighted projective space $\CP_{a}^{n+1}$ and the projective bundle $P(\CC\oplus\gamma^{ a})$ over $\CP^n$ is provided by the following lemma.

    \begin{lemma}\label{lem:blow-up}
        Let $a>1$. Then, a projective bundle $P(\CC\oplus\gamma^{a})$ over $\CP^n$ is the blow-up of $\CP_a^{n+1}$ at the singular point.
    \end{lemma}
    \begin{proof}
        Each one-dimensional cone in a fan of $P(\CC\oplus\gamma^{ a})$ is generated by one of the elements in $$S:=\{-\mathbf{e}_1-\cdots-\mathbf{e}_n-a\mathbf{e}_{n+1}, -\mathbf{e}_{n+1}, \mathbf{e}_1,\ldots,\mathbf{e}_{n+1}\}.$$ We obtain a fan of $\CP_{a}^{n+1}$ by taking the cones generated by all proper subsets of the set $S\setminus\{-\mathbf{e}_{n+1}\}$, see Figure~\ref{fig:hirzebruch is a blow-up}. Note that the cone generated by the set $\{-\mathbf{e}_1-\cdots-\mathbf{e}_n-a\mathbf{e}_{n+1}, \mathbf{e}_1,\ldots,\mathbf{e}_{n}\}$ corresponds to the singular point $[0,\ldots,0,1]$ of $\CP^{n+1}_a$, and $-a\mathbf{e}_{n+1}=(-\mathbf{e}_1-\cdots-\mathbf{e}_n-a\mathbf{e}_{n+1})+ \mathbf{e}_1+\cdots +\mathbf{e}_{n}$. Hence, $P(\CC\oplus\gamma^{a})$ is the blow-up of $\CP_{a}^{n+1}$ at the singular point. See Figure~\ref{fig:hirzebruch is a blow-up}.
        \begin{figure}[h]
            \begin{center}
            \begin{tikzpicture}[scale=.8]
                \foreach \x in {0,0.2,...,1}
                {
                    \draw (2.5+\x,2.5)--(2.5,2.5+\x);
                    \draw (2.5+\x,2.5)--(2.5,2.5-\x);
                    \draw (2.5,2.5+\x)--(2.5-2*\x,2.5-2*\x);
                    \draw (2.5,2.5-\x)--(2.5-2*\x,2.5-2*\x);
                }
                \draw[->,thick] (2.5,2.5)--(2.5,3.6);
                \draw[->,thick] (2.5,2.5)--(3.6,2.5);
                \draw[->,thick] (2.5,2.5)--(2.5,1.4);
                \draw[->,thick] (2.5,2.5)--(0.4,0.4);
                \draw (2.5,4) node{\tiny{$\begin{pmatrix}0\\1\end{pmatrix}$}};
                \draw (4,2.5) node{\tiny{$\begin{pmatrix}1\\0\end{pmatrix}$}};
                \draw (2.5,0.9) node{\tiny{$\begin{pmatrix}0\\-1\end{pmatrix}$}};
                \draw (-0.1,0.5) node{\tiny{$\begin{pmatrix}-1\\-a\end{pmatrix}$}};
                \draw (5,2.5) node{$\longleftarrow$};
                \foreach \x in {0,0.2,...,1}
                {
                    \draw (7.5+\x,2.5)--(7.5,2.5+\x);
                    \draw (7.5+\x,2.5)--(7.5-2*\x,2.5-2*\x);
                    \draw (7.5,2.5+\x)--(7.5-2*\x,2.5-2*\x);
                }
                \draw[->,thick] (7.5,2.5)--(7.5,3.6);
                \draw[->,thick] (7.5,2.5)--(8.6,2.5);
                \draw[->,thick] (7.5,2.5)--(5.4,0.4);
                \draw (7.5,4) node{\tiny{$\begin{pmatrix}0\\1\end{pmatrix}$}};
                \draw (9,2.5) node{\tiny{$\begin{pmatrix}1\\0\end{pmatrix}$}};
                \draw (4.9,0.5) node{\tiny{$\begin{pmatrix}-1\\-a\end{pmatrix}$}};
            \end{tikzpicture}
            \end{center}
            \caption{A Hirzebruch surface is a blow-up from $\CP^2_{(1,1,a)}$.}\label{fig:hirzebruch is a blow-up}
        \end{figure}
    \end{proof}

    \begin{lemma}\label{lem:wps connected sum}
        \begin{enumerate}
            \item For $a>0$, $M_{a,\mathbf{0}}$ is homeomorphic to $\CP^{n+1}_a\#\overline{\CP^{n+1}_a}$;
            \item Let $\vr=(\underbrace{1,\ldots,1}_r,0,\ldots,0)\in \Z^n$. Then, $M_{2,\vr}$ is homeomorphic to
                \begin{enumerate}
                    \item $\CP_2^{n+1}\#\overline{\CP_2^{n+1}}$ if $r$ is even;
                    \item $\CP_2^{n+1}\#\CP_2^{n+1}$ if $r$ is odd.
                \end{enumerate}
        \end{enumerate}
    \end{lemma}
    \begin{proof}
We note that  a blow-up of $\CP^{n+1}_a$ at the singular point is indeed $\CP^{n+1}_a\#\overline{\CP^{n+1}_a}$. It follows from Lemma~\ref{lem:blow-up} that $M_{a,\mathbf{0}}$ is equivalent to $\CP^{n+1}_a\#\overline{\CP^{n+1}_a}$ by comparing their characteristic functions. Hence, statement (1) is proved.

        Now let us prove statement (2).
        Let $r$ be the number of nonzero components of $\vr$. Then, as in \eqref{eq:char mx of M}, the characteristic matrix of $M_{2,\vr}$ is
        \begin{equation*}
        \Lambda=
        \begin{pmatrix}
                -1 &-1&1&      & & &      &&    \\
                 \vdots&\vdots&&\ddots& & &      && \\
                 -1&-1&&     &1& &      & &    \\
                 -1&0&&     & &1&      & &    \\
                  \vdots&\vdots&&    & & &\ddots& &\\
                  -1&0&&    & & &      &1&     \\
                  -2&-1&&    & & &      & &1
            \end{pmatrix}
        \end{equation*}
        where the order in $\mathfrak{F}(\Delta^n \times \Delta^1)$ is
        $$ F_1 \times \Delta^1, \Delta^n \times F_1', F_2\times\Delta^1,\ldots,F_{n+1} \times \Delta^1,\Delta^n \times F_2'.
        $$
        We note that its orbit space $\Delta^n \times \Delta^1$ can be identified with $\Delta^{n+1} \# \Delta^{n+1}$. The columns of $\Lambda$ corresponding to the (ordered) subset
        $\{F_1 \times \Delta^1, \ldots, F_{n+1}\times\Delta^1,\Delta^n \times F_1'\}$ form the rational characteristic matrix $\widetilde\Lambda$ on $\Delta^{n+1}$
        $$
            \widetilde{\Lambda} = \begin{pmatrix}
                -1&1&      & & &      & &-1    \\
                 \vdots&&\ddots& & &      & &\vdots\\
                 -1&&     &1& &      & &-1    \\
                 -1&&     & &1&      & &0     \\
                 \vdots&&     & & &\ddots& &\vdots\\
                 -1&&    & & &      &1&0     \\
                 -2&&    & & &      & &-1
            \end{pmatrix}.
        $$ Moreover, the columns of $\Lambda$ corresponding to the subset $\{F_1 \times \Delta^1, \ldots, F_{n+1}\times\Delta^1,\Delta^n \times F_2'\}$ also form the rational characteristic matrix $\Lambda_2$ on $\Delta^{n+1}$
        $$
            \Lambda_2 = \begin{pmatrix}
                -1&1&      & & &      & &    \\
                 \vdots&&\ddots& & &      & &\\
                 -1&&     &1& &      & &    \\
                 -1&&     & &1&      & &     \\
                 \vdots&&     & & &\ddots& &\\
                 -1&&    & & &      &1&     \\
                 -2&&    & & &      & &1
            \end{pmatrix}.
        $$
        Note that the first $n+1$ columns of $\widetilde\Lambda$ are equal to those of $\Lambda_2$, and they have the determinant $2$. Furthermore, the characteristic matrix corresponding to $Q(\Delta^{n+1}, \widetilde{\Lambda})\#Q(\Delta^{n+1}, \Lambda_2)$ coincides with $\Lambda$.
        Hence, $M_{2,\vr}$ is equivalent to the equivariant connected sum of $Q(\Delta^{n+1}, \widetilde{\Lambda})$ and  $Q(\Delta^{n+1}, \Lambda_2)$ at the singular points. Note that
        %
        \begin{equation}\label{eq:mx multiplication}
        \widetilde{\Lambda}=\begin{pmatrix}
                1&      & & &      & &1    \\
                 &\ddots& & &      & &\vdots\\
                 &     &1& &      & &1    \\
                 &     & &1&      & &0     \\
                  &    & & &\ddots& &\vdots\\
                  &    & & &      &1&0     \\
                  &    & & &      & &1
            \end{pmatrix}\begin{pmatrix}
                 1&1&      & & &      & &0    \\
                 \vdots&&\ddots& & &      & &\vdots\\
                   1&&    &1& &      & &0    \\
                  -1&&    & &1&      & &0     \\
                  \vdots&&    & & &\ddots& &\vdots\\
                  -1&&    & & &      &1&0     \\
                  -2&&    & & &      & &-1
            \end{pmatrix}.
        \end{equation}
        Since the first matrix on the right hand side of~\eqref{eq:mx multiplication} is in $\mathrm{GL}(n+1,\Z)$,
        $Q(\Delta^{n+1},\widetilde{\Lambda})$ is equivalent to $\CP^{n+1}_2$. Therefore, $M_{2,\vr}$ is equivalent to either $\CP^{n+1}_2 \# \overline{\CP^{n+1}_2}$ or $\CP^{n+1}_2 \# \CP^{n+1}_2$.

        In particular, according to Theorem 5.5 of \cite{CPS12},
        \begin{itemize}
            \item if $n$ is even, $M_{2,\vr}=M_{2,\mathbf{0}} = \CP^{n+1}_2 \# \overline{\CP^{n+1}_2} = \CP^{n+1}_2 \# \CP^{n+1}_2$;
            \item if $n$ is odd,
            \begin{itemize}
                \item $M_{2,\vr}=M_{2,\mathbf{0}}= \CP^{n+1}_2 \# \overline{\CP^{n+1}_2}$ for even $r$, and
                \item $M_{2,\vr}=M_{2,(1,0,\ldots,0)} = \CP^{n+1}_2 \# \CP^{n+1}_2$ for odd $r$.
            \end{itemize}
         \end{itemize}
    This proves statement (2).
    \end{proof}

    Now, let us consider the cohomology ring of $\CP_a^{n+1}\#\overline{\CP_{a}^{n+1}}$.
    Letting $X:=\CP^{n+1}_a\setminus D^\circ$ and
    $Y:=\overline{\CP^{n+1}_a}\setminus {D'}^\circ$, we can see that $X\cup Y=\CP_a^{n+1}\#\overline{\CP_{a}^{n+1}}$ and $X\cap Y=\partial D=\partial D'$, where $D^\circ$ and ${D'}^\circ$ are the interiors of the closed balls $D$ and $D'$ containing the singular point in $\CP^{n+1}_a$ and $\overline{\CP^{n+1}_a}$, respectively.

    Let $\tilde{u}$ and $\tilde{v}$ be the elements of $H_{2n}(X)$ and $H_{2n}(Y)$, respectively, represented by the submanifolds $\CP^n=\{z_{n+1}=0\} \subset X$ and $\overline{\CP^n}=\{z_{n+1}=0\}\subset Y$, respectively. Thus, $\tilde{u}$ and $\tilde{v}$ are elements of $H_{2n}(\CP^{n+1}_a\#\overline{\CP^{n+1}_a})$. See  Figure~\ref{fig:connected sum}.
        \begin{figure}[h]
            \begin{center}
                \begin{tikzpicture}[scale=.8]
                    \draw (1.5,1)--(3,2)--(1.5,3)--cycle;
                    \draw (5,2)--(6.5,1)--(6.5,3)--cycle;
                    \draw (9.5,1)--(11.5,1)--(11.5,3)--(9.5,3)--cycle;
                    \draw (4,2) node{$\#$};
                    \draw (8,2) node{$=$};
                    \draw (2,2) node{$\circlearrowleft$};
                    \draw (6,2) node{$\circlearrowleft$};
                    \draw (10.5,2) node{$\circlearrowleft$};
                    \draw (2.5,3) node{\tiny{$\begin{pmatrix}-1\\-a\end{pmatrix}$}};
                    \draw (2.5,1.1) node{\tiny{$\begin{pmatrix}1\\0\end{pmatrix}$}};
                    \draw (1.1,2) node{\tiny{$\begin{pmatrix}0\\1\end{pmatrix}$}};
                    \draw (5.5,3) node{\tiny{$\begin{pmatrix}-1\\-a\end{pmatrix}$}};
                    \draw (5.5,1.1) node{\tiny{$\begin{pmatrix}1\\0\end{pmatrix}$}};
                    \draw (7,2) node{\tiny{$\begin{pmatrix}0\\-1\end{pmatrix}$}};
                    \draw (10.5,3.5) node{\tiny{$\begin{pmatrix}-1\\-a\end{pmatrix}$}};
                    \draw (10.5,0.5) node{\tiny{$\begin{pmatrix}1\\0\end{pmatrix}$}};
                    \draw (9.1,2) node{\tiny{$\begin{pmatrix}0\\1\end{pmatrix}$}};
                    \draw (12,2) node{\tiny{$\begin{pmatrix}0\\-1\end{pmatrix}$}};
                    \fill (3,2) circle(2pt);
                    \fill (5,2) circle(2pt);
                    \draw[ultra thick,red] (1.5,1)--(1.5,3);
                    \draw[ultra thick,blue] (6.5,1)--(6.5,3);
                    \draw[ultra thick,red] (9.5,1)--(9.5,3);
                    \draw[ultra thick,blue] (11.5,1)--(11.5,3);
                    \draw (0.5,0.5) node{$\tilde{u}$};
                    \draw[<->,snake=snake,segment amplitude=.4mm,segment length=2mm,line after snake=1mm] (0.7,0.75) -- (1.3,1.3);
                    \draw[<->,snake=snake,segment amplitude=.4mm,segment length=2mm,line after snake=1mm] (6.7,1.3) -- (7.3,0.75);
                    \draw (7.5,0.5) node{$\tilde{v}$};
                    \draw (8.5,0.5) node{$\tilde{u}$};
                    \draw[<->,snake=snake,segment amplitude=.4mm,segment length=2mm,line after snake=1mm] (8.7,0.75) -- (9.3,1.3);
                    \draw[<->,snake=snake,segment amplitude=.4mm,segment length=2mm,line after snake=1mm] (11.7,1.3) -- (12.3,0.75);
                    \draw (12.5,0.5) node{$\tilde{v}$};
                \end{tikzpicture}
            \end{center}
            \caption{$\tilde{u}$ and $\tilde{v}$ in $H_2(\CP_a^2\#\overline{\CP_a^2})$} \label{fig:connected sum}
        \end{figure}

    Recall the characteristic matrix~\eqref{eq:char mx of connected sum} of $\CP^{n+1}_a\#\overline{\CP^{n+1}_a}$. Then the cohomology ring is
    $$
    H^\ast(\CP^{n+1}_a\#\overline{\CP^{n+1}_a})=\Z[x_1,x_2]/\langle x_1^{n+1},x_2(ax_1 +x_2)\rangle,
    $$
    where $x_1$ and $x_2$ correspond to the first and the second columns of~\eqref{eq:char mx of connected sum}, respectively. Because $\tilde{u}$ and $\tilde{v}$ represent the characteristic submanifolds associated with $\Delta^n\times F_1'$ and $\Delta^n\times F_2'$, through the Poincar\'{e} duality, $\tilde{u}$ and $\tilde{v}$ correspond to $u=ax_1+x_2$ and $v=x_2$ in $H^2(\CP^{n+1}_a\#\overline{\CP^{n+1}_a})$, where the identities originate from $\mathcal{J}$ in~\eqref{eq:cohomology of quasitoric manifold}.

    Now assume that $ab=2$. By Lemma~\ref{lem:wps connected sum}, $\CP^{n+1}_a\#\overline{\CP^{n+1}_a} \cong M_{a,\mathbf{0}}$ and $\CP^{n+1}_a\#\CP^{n+1}_a \cong M_{a,(b,0,\ldots,0)}$. Hence, by using the cohomology formula \eqref{eq:cohomology of M}, we compute their cohomology rings as follows:
    \begin{equation*}
        \begin{split}
            H^\ast(\CP^{n+1}_a\#\overline{\CP^{n+1}_a})&=\Z[x_1,x_2]/\langle x_1^{n+1},x_2(ax_1+x_2)\rangle,\mbox{ and}\\
            H^\ast(\CP^{n+1}_a\#\CP^{n+1}_a)&=\Z[x_1,x_2]/\langle x_1^n(x_1+bx_2), x_2(ax_1+x_2)\rangle.
        \end{split}
    \end{equation*}
  Note that $u=ax_1+x_2$ and $v=x_2$ correspond to the $(n+1)$th and $(n+3)$th columns of the characteristic matrix
        \begin{equation*}
        \begin{pmatrix}
                -1 &0&1&&            &&    \\
                -1 &0&                &1&            &&    \\
                 \vdots&\vdots& &&    \ddots& &\\
                -1&0&  &&          &1&     \\
                 -a&-1& &&          & &1
            \end{pmatrix}\text{ or }
        \begin{pmatrix}
                -1 &-b&1&&            &&    \\
                -1 &0&&1&            &&    \\
                \vdots&\vdots& &&\ddots      && \\
                -1&0&  &&          &1&     \\
                -a&-1&  &&          & &1
            \end{pmatrix}.
        \end{equation*}

    \begin{proposition}\label{prop:realizable a=1,2}
        Assume $a=1$ or $a=2$. Then, the ring automorphism groups $\Aut(H^\ast(\CP^{n+1}_a\#\overline{\CP^{n+1}_a}))$ and $\Aut(H^\ast(\CP^{n+1}_a\#\CP^{n+1}_a))$ are realizable by diffeomorphisms.
    \end{proposition}
    \begin{proof}
        If $n=1$, then $\CP^{2}_a\#\overline{\CP^{2}_a}$ is a Hirzebruch surface, and $\CP^2_2 \# \CP^2_2$ is diffeomorphic to $\CP^2 \# \CP^2$. According to \cite{CM} or \cite{Wall64}, all ring automorphisms on their cohomology rings are realizable by diffeomorphisms. Henceforth, let us assume that $n>1$.

        We first compute the ring automorphism groups of $H^\ast(\CP^{n+1}_a\#\overline{\CP^{n+1}_a})$ and $H^\ast(\CP^{n+1}_a\#\CP^{n+1}_a)$ as subgroups of $\GL(2,\Z)$ (see Remark~\ref{rem:Aut_in_GL}). For each case, there is only one relation $x_2(ax_1 + x_2)=0$ such that a product of two degree-two elements is zero up to scalar multiplication. Accordingly,
        an automorphism should send $\{x_2, ax_1 + x_2\}$ to $\{x_2, ax_1 + x_2\}$ up to sign. Hence, there are at most $8$ automorphisms.

       Let $u=ax_1+x_2$ and $v=x_2$. Then, we have $u^{n+1}=(-v)^{n+1}$ in $H^\ast(\CP^{n+1}_a\#\overline{\CP^{n+1}_a})$ and $u^{n+1}=(-1)^nv^{n+1}$ in $H^\ast(\CP^{n+1}_a\#\CP^{n+1}_a)$.
        If $n$ is even, neither $H^\ast(\CP^{n+1}_a\#\overline{\CP^{n+1}_a})$ nor $H^\ast(\CP^{n+1}_a\#\CP^{n+1}_a)$ has any automorphism $(u,v)\mapsto \pm(u,-v)$. Hence,
         \begin{equation*}
            \begin{split}
                &\Aut(H^\ast(\CP^{n+1}_a\#\overline{\CP^{n+1}_a}))\\
                &=\left\{ \begin{pmatrix}1&0\\0&1\end{pmatrix}, \begin{pmatrix}-1&0\\0&-1\end{pmatrix}, \begin{pmatrix}1&0\\-a&-1\end{pmatrix}, \begin{pmatrix}-1&0\\a&1\end{pmatrix}\right\}\\
                &=\Aut(H^\ast(\CP^{n+1}_a\#\CP^{n+1}_a))\\
                &\cong(\Z_2)^2.
            \end{split}
        \end{equation*}
        If $n$ is odd, then
        \begin{equation*}
        \begin{split}
            &\Aut(H^\ast(\CP^{n+1}_a\#\overline{\CP^{n+1}_a}))\\
            &= \left\{ \begin{pmatrix}1&0\\0&1\end{pmatrix}, \begin{pmatrix}-1&0\\0&-1\end{pmatrix},
            \begin{pmatrix}1&\frac{2}{a}\\0&-1\end{pmatrix}, \begin{pmatrix}-1&-\frac{2}{a}\\0&1\end{pmatrix},\right.\\
            &\qquad\left. \begin{pmatrix}-1&0\\a&1\end{pmatrix}, \begin{pmatrix}1&0\\-a&-1\end{pmatrix}, \begin{pmatrix}1&\frac{2}{a}\\-1&-1\end{pmatrix}, \begin{pmatrix}-1&-\frac{2}{a}\\a&1\end{pmatrix}\right\}\\
            &=\Aut(H^\ast(\CP^{n+1}_a\#\CP^{n+1}_a))\\
            &\cong(\Z_2)^3.
        \end{split}
        \end{equation*}

        We consider an involution $s$ on $\CP^{n+1}_a$ defined by
        \begin{equation*}
                s \colon [z_0,\ldots,z_{n+1}]\mapsto [\overline{z_0},\ldots,\overline{z_{n+1}}].
        \end{equation*}
        For odd $n$, we consider another involution $t$ defined by
        $$
                t \colon [z_0,\ldots,z_{n+1}]\mapsto [-z_0,\ldots,-z_{k-1},z_k,\ldots,z_{n+1}],
        $$
         where $k=\frac{n+1}{2}$. Observe that
        \begin{enumerate}
            \item the involution $s$ reverses the orientation of the submanifold $\CP^n=\{z_{n+1}=0\}$, and fixes the real weighted projective space $\RP_a^{n+1}$;
            \item the fixed point set of the involution $t$ is the disjoint union of $\{z_{k}=\cdots=z_{n+1}=0\}=\CP^{k-1}$ and $\{z_0=\cdots=z_{k-1}=0\}=\CP_a^{k};$
            \item the point $[0,\ldots,0,1]$ is fixed by both $s$ and $t$.
        \end{enumerate}

        Note that if $a=1$, then $[0,\ldots,0,1]$ is a smooth point. If $a=2$, both $[0,\ldots,0,1]\in \RP_a^{n+1}$ and $[0,\ldots,0,1]\in\CP_a^k$ have the same singularity, that is, $[0,\ldots,0,1]\in \RP_a^{n+1}$ is locally modeled by $\R^{n+1}/\mu_2$, and $[0,\ldots,0,1]\in \CP_a^k$ is locally modeled by $\C^k/\mu_2$.

        \textbf{Type 1.} We consider the involution $s$ on both $\CP_a^{n+1}$ and $\overline{\CP_a^{n+1}}$. Take the equivariant connected sum of $\CP_a^{n+1}$ and $\overline{\CP_a^{n+1}}$ at $[0,\ldots,0,1]$. Then the resulting involution on $\CP_a^{n+1}\#\overline{\CP_a^{n+1}}$ sends $(u,v)$ to $(-u,-v)$.

        \textbf{Type 2.} We consider the involution $s$ on $\CP_a^{n+1}$ and $t$ on $\overline{\CP_a^{n+1}}$. Take the equivariant connected sum of $\CP_a^{n+1}$ and $\overline{\CP_a^{n+1}}$ at $[0,\ldots,0,1]$. Then the resulting involution on $\CP_a^{n+1}\#\overline{\CP_a^{n+1}}$ sends $(u,v)$ to $(-u,v)$.

        \textbf{Type 3.} Let $D$ be a sub-orbifold with a boundary which is diffeomorphic to $D^{2(n+1)}/\mu_a$, where $D^{2(n+1)}$ is the closed unit ball. Then $\CP_a^{n+1}\#\overline{\CP_{a}^{n+1}}$ is obtained by deleting the interiors of sub-orbifolds $D$ containing $[0,\ldots,0,1]$ from $\CP_{a}^{n+1}$ and $\overline{\CP_{a}^{n+1}}$, respectively, and gluing together the resulting boundaries $\partial D$. Hence, $\CP_a^{n+1}\#\overline{\CP_{a}^{n+1}}$ admits a reflection about the boundary $\partial D$ which maps $\CP_a^{n+1}\setminus D$ to $\overline{\CP_{a}^{n+1}}\setminus D$. This reflection sends $(u,v)$ to $(v,u)$.

        Then Type 1 corresponds to $\begin{pmatrix}
            -1&0\\0&-1
        \end{pmatrix}$, Type 2 corresponds to $\begin{pmatrix}
            -1&-\frac{2}{a}\\0&1
        \end{pmatrix}$, and Type 3 corresponds to $\begin{pmatrix}
            -1&0\\a&1
        \end{pmatrix}$.

        Combining the diffeomorphisms of the three types above, it becomes possible to realize every element of $\Aut(H^\ast(\CP_{a}^{n+1}\#\overline{\CP_{a}^{n+1}}))$ for $a=1,2$.

        Because we can use the same process for $\CP_{a}^{n+1}\#\CP_{a}^{n+1}$, we can realize every element of $\Aut(H^\ast(\CP_{a}^{n+1}\#\CP_{a}^{n+1}))$ for $a=1,2$.
    \end{proof}

    \begin{remark}
        In general, if $a$ is odd, $[0,\ldots,0,1]\in \RP_{a}^{n+1}$ is smooth, and if $a$ is even,  $[0,\ldots,0,1]\in \RP_{a}^{n+1}$ has an isotropy group $\mu_2$. On the other hand, $[0,\ldots,0,1]\in \CP_{a}^k$ has an isotropy group $\mu_a$. Hence, Type~$2$ is only possible when $a=1$ or $2$.
    \end{remark}

    Combining Lemma \ref{lem:wps connected sum} and Proposition \ref{prop:realizable a=1,2}, we obtain the following.
    \begin{corollary}\label{cor:realizable a=1,2}
        Assume $a=1$ or $2$. Then for a quasitoric manifold $M_{a,\vb}$, every automorphism of $H^\ast(M_{a,\vb})$ is realizable by a homeomorphism.
    \end{corollary}

\section{Two-stage generalized Bott manifolds}\label{sec:two-stage GB}

    In this section, we restrict our attention to two-stage generalized Bott manifolds. We show that every cohomology ring automorphism of a two-stage generalized Bott manifold is realizable by a diffeomorphism. We prepare the following two lemmas.

    \begin{lemma}\label{lem:total chern determine cb over CP}[Lemma 5.2 \cite{ch-ma-su10b}]
        Let $E$ and $E'$ be Whitney sums of complex line bundles over the complex projective space $\CP^n$ of the same dimension. If $E$ and $E'$ have the same total Chern classes, then $E$ and $E'$ are isomorphic.
    \end{lemma}

    \begin{lemma}\label{lem:two-stage generalized Bott}[Lemma 6.2 \cite{ch-ma-su10b}]
        Let $M=P(\CC\oplus\bigoplus_{i=1}^m\gamma^{a_i})$ and $M'=P(\CC\oplus\bigoplus_{i=1}^m\gamma^{a_i'})$ be projective bundles over the complex projective space $B=\CP^n$. Assume that $m$ is greater than $1$, then every cohomology ring isomorphism $\varphi\colon H^\ast(M)\to H^\ast(M')$ preserves the subring $H^\ast(B)$ unless $M$ is $\CP^n\times\CP^m$.
    \end{lemma}

    Now we can show the realizability of a cohomology ring automorphism of a two-stage generalized Bott manifold.
    \begin{proposition}\label{prop:two-stage GB}
        Let $E:=\CC\oplus\bigoplus_{j=1}^m\gamma^{a_j}$ be the Whitney sum of complex line bundles over $\CP^n$.  Then, every graded ring automorphism of $H^\ast(P(E))$ is induced by a diffeomorphism.
    \end{proposition}
    \begin{proof}
        Since every cohomology ring automorphism of a product of complex projective spaces is induced by a diffeomorphism \cite{Choi-Suh}, we may assume that $P(E)$ is a non-trivial fiber bundle.

        Note that $P(E)$ is a Hirzebruch surface if $n=m=1$. Each cohomology ring automorphism of a Hirzebruch surface is realizable by a diffeomorphism by \cite{CM} or \cite{Wall64}.

        If $m=1$ and $1\leq a_1\leq2$, then $P(E)$ is diffeomorphic to $\CP^{n+1}_{a_1}\#\overline{\CP^{n+1}_{a_1}}$. Hence, by Proposition \ref{prop:realizable a=1,2}, every automorphism of $H^\ast(P(E))$ is realizable by a diffeomorphism.

        Then, the remaining cases are (i) $m>1$ and (ii) $m=1$, $n>1$, and $a_1>2$. Note that
        \begin{align*}
        H^\ast(P(E)) &=
        H^\ast(\CP^n)[x_2]\left/ \left\langle x_2\prod_{j=1}^m\left(a_jx_1+x_2\right) \right\rangle  \right. \\
        &= \Z[x_1,x_2]\left/\left\langle x_1^{n+1},~x_2\prod_{j=1}^m\left(a_jx_1+x_2\right)\right\rangle\right.,
        \end{align*}
        where $x_1=-c_1(\gamma)\in H^2(\CP^n)\subset H^2(P(E))$ and $x_2\in H^2(P(E))$ is the negative of the first Chern class of the tautological line bundle over $P(E)$. We first claim that every cohomology ring automorphism of $H^\ast(P(E))$ preserves the subring $H^\ast(\CP^n)$ in each case.

        In the first case, i.e., if $m>1$, by Lemma~\ref{lem:two-stage generalized Bott}, every automorphism of $H^\ast(P(E))$ preserves the subring $H^\ast(\CP^n)$.

        Now, we consider the second case, i.e., $m=1$, $n>1$, and $a_1>2$. Let $\varphi$ be a ring automorphism of $H^\ast(P(E))$. Since $n>1$, there is only one relation $x_2(a_1x_1+x_2)=0$ such that a product of two degree-two elements is zero up to scalar multiplication. Thus, $\varphi$ should send $\{x_2,a_1x_1+x_2\}$ to $\{x_2,a_1x_1+x_2\}$ up to sign. Suppose $\varphi(a_1x_1+x_2)=\pm(a_1x_1+x_2)$ and $\varphi(x_2)=\mp x_2$. Then $\varphi(x_1)=\pm(x_1+\frac{2}{a_1}x_2)$. Because $a_1>2$, $\varphi$ cannot be an isomorphism. Therefore, there are only four automorphisms of $H^\ast(P(\CC\oplus\gamma^{a_1}))$ as follows:
        \begin{equation*}
            \Aut(H^\ast(P(E)))=\left\{
            \begin{pmatrix}1&0\\0&1\end{pmatrix}
            \begin{pmatrix}-1&0\\0&-1\end{pmatrix}
            \begin{pmatrix}1&0\\-a_1&-1\end{pmatrix}
            \begin{pmatrix}-1&0\\a_1&1\end{pmatrix}\right\}.
        \end{equation*}
        Hence, in each case, every ring automorphism of $H^\ast(P(E))$ preserves the subring $H^\ast(\CP^n)$, which proves the claim.

        Let $\varphi$ be a ring automorphism of $H^\ast(P(E))$. By the above claim, $\varphi(x_1)=\pm x_1$. Since every automorphism of $H^\ast(\CP^n)$ is induced by a diffeomorphism, we may assume that $\varphi(x_1)=x_1$.

        We write $\varphi(x_2)=\epsilon x_2+Ax_1$, where $\epsilon=\pm1$ and $A\in\Z$.
        Then the map $\varphi$ lifts to a grading preserving isomorphism $\overline{\varphi}\colon \Z[x_1,x_2]\to \Z[x_1,x_2]$ with $\overline{\varphi}(\tilde{\mathcal{J}})=\tilde{\mathcal{J}}$, where $\tilde{\mathcal{J}}\subset \Z[x_1,x_2]$ is the ideal generated by the homogeneous polynomials $x_1^{n+1}$ and $x_2\prod_{j=1}^m (a_jx_1+x_2)$.

        (I) We assume that $\varphi(x_2)=x_2+Ax_1$.
        Since $\overline{\varphi}(x_2\prod_{j=1}^m(a_jx_1+x_2)) \in \tilde{\mathcal{J}}$,
        we have
        \begin{equation}\label{eq:epsilon 1}
            (x_2+Ax_1)\prod_{j=1}^m(x_2+(A+a_j)x_1)=f(x_1,x_2)x_1^{n+1}+\alpha x_2\prod_{j=1}^m(x_2+a_j x_1),
        \end{equation}
        where $f(x_1,x_2)$ is a homogeneous polynomial of degree $m-n$ and $\alpha$ is an integer. Note that if $n\geq m$, then $f=0$. By comparing the coefficients of $x_2^{m+1}$ and $x_1x_2^m$ on both sides of~\eqref{eq:epsilon 1},
        it is clear that $A=0$. Hence, $\varphi$ is the identity that is obviously induced from the identity map of $P(E)$.

        (II) Now assume that $\varphi(x_2)=-x_2+Ax_1$.
        Since $\overline{\varphi}(x_2\prod_{j=1}^m(a_jx_1+x_2))$ belongs to $\tilde{\mathcal{J}}$, we have
        \begin{equation}\label{eq:epsilon -1}
            (-x_2+Ax_1)\prod_{j=1}^m(-x_2+(A+a_j)x_1)=f(x_1,x_2)x_1^{n+1}+\alpha x_2\prod_{j=1}^m(x_2+a_j x_1),
        \end{equation}
        where $f(x_1,x_2)$ is a homogeneous polynomial of degree $m-n$ and $\alpha$ is an integer. By comparing the coefficients of $x_2^{m+1}$ on both sides of~\eqref{eq:epsilon -1}, it follows that $\alpha=(-1)^{m+1}$.
        By substituting $x_2=1$ into \eqref{eq:epsilon -1}, we obtain
        \begin{equation}\label{eq:compare total chern}
            (1-Ax_1)(1-(A+a_1)x_1)\cdots(1-(A+a_m)x_1)=(1+a_1x_1)\cdots(1+a_mx_1)
        \end{equation} in $H^\ast(\CP^n)=\Z[x_1]/\langle x_1^{n+1}\rangle$.
        Since $E$ possesses a Hermitian metric, its dual bundle $E^\ast=\Hom(E,\C)$ is canonically isomorphic to the conjugate bundle $\CC\oplus\gamma^{-a_1}\oplus\cdots\oplus\gamma^{-a_m}$. By Lemma \ref{lem:total chern determine cb over CP}, equation \eqref{eq:compare total chern} implies that
        $$E^\ast\otimes\gamma^{-A}=\gamma^{-A}\oplus\gamma^{-A-a_1}\oplus\cdots\oplus\gamma^{-A-a_m} =\CC\oplus\gamma^{a_1}\oplus\cdots\oplus\gamma^{a_m}=E$$ as complex vector bundles over $\CP^n$.

        Let $<~\,,~>$ be a Hermitian metric on $E$. Then the map $\widetilde{h}\colon E\to E^\ast$, $u\mapsto<u,\cdot>$, induces the isomorphism $h\colon P(E)\to P(E^\ast)$ as fiber bundles.  If $y$ is the negative of the first Chern class of the tautological line bundle over $P(E^\ast)$, then $h^\ast(y)=-x_2$.

        For each $q\in \CP^n$, we choose a non-zero vector $v_q$ from the fiber of $\gamma^{-A}$ over $q$ and define a map $\widetilde{g}\colon E^\ast\to E^\ast\otimes\gamma^{-A}$ by $\widetilde{g}(u_q)=u_q\otimes v_q$, where $u_q$ is an element of the fiber of $E^\ast$ over $q$. The map $\widetilde{g}$ depends on the choice of $v_q$'s but the induced map $g\colon P(E^\ast)\to P(E^\ast\otimes\gamma^{-A})$ does not have this dependency because $\gamma^{-A}$ is a line bundle. Then the map
        $$g\colon P(E^\ast)\to P(E^\ast\otimes\gamma^{-A})=P(E)$$ preserves the complex structures on each fiber. Therefore, it induces a complex vector bundle isomorphism $T_fP(E^\ast)\to T_fP(E^\ast\otimes\gamma^{-A})$ between their tangent bundles along the fibers. According to the Borel-Hirzebruch formula, their respective total Chern classes are
        $$(1+y)(1-a_1x_1+y)\cdots(1-a_mx_1+y)$$ and $$(1-Ax_1+x_2)(1-Ax_1-a_1x_1+x_2)\cdots(1-Ax_1-a_mx_1+x_2).$$ Since $g^\ast(c_1(T_f(P(E))))=c_1(T_f(P(E^\ast)))$, we have
        $$g^\ast\left((m+1)(x_2-Ax_1)-\sum_{j=1}^ma_jx_1\right)=(m+1)y-\sum_{j=1}^ma_jx_1.$$
        Further, the map $g$ covers the identity map on $\CP^n$; thus, $g^\ast(x_2)=y+Ax_1$. Therefore,
        $$h^\ast(g^\ast(x_2))=-x_2+Ax_1=\varphi(x_2).$$

        By (I) and (II), every ring automorphism $\varphi$ is induced by a diffeomorphism.
    \end{proof}

\section{Quasitoric manifolds over $\Delta^n\times\Delta^m$}\label{sec:quasitoric over n x m}

    In this section, we show that every cohomology ring automorphism of a quasitoric manifold with second Betti number 2 is realizable by a homeomorphism. As we have seen in the previous section, every cohomology ring automorphism of a two-stage generalized Bott manifold is realizable by a diffeomorphism. Hence, we only need to consider quasitoric manifolds over $\Delta^n\times\Delta^m$ which are not equivalent to a two-stage generalized Bott manifold.

    Let $M_{\va,\vb}$ be a quasitoric manifold over $\Delta^n\times\Delta^m$.
    By Theorem~\ref{thm:main of pjm}, it is sufficient to consider the case when $\va = \vs = (2,\ldots,2,0,\ldots,0) \neq \mathbf{0}$ and $\vb=\vr=(1,\ldots,1,0,\ldots,0) \neq \mathbf{0}$.

    \begin{proposition}\label{prop:NGB}
        Let $M_{\vs,\vr}$ be a quasitoric manifold over $\Delta^n\times \Delta^m$, where two nonzero vectors $\vs$ and $\vr$ have the forms
        \begin{equation*}
            \vs:=(\underbrace{2,\ldots,2}_s,0,\ldots,0)\in\Z^m \mbox{ and } \vr:=(\underbrace{1,\ldots,1}_r,0,\ldots,0)\in\Z^n.
        \end{equation*}
        Then every element of $\Aut(H^\ast(M_{\vs,\vr}))$ is induced by a homeomorphism.
    \end{proposition}
    \begin{proof}
        The detailed computation of $\Aut(H^\ast(M_{\vs,\vr}))$ can be found in the proof of Theorem~6.2 in \cite{CPS12}. Even though it is one of key parts of this proof, the result is used here without detailed calculation to avoid repetition of the elementary computation.

        If $n=1$ or $m=1$,
        every automorphism of $H^\ast(M_{\vs,\vr})$ is realizable by a homeomorphism by Corollary \ref{cor:realizable a=1,2}.

        Now, assume that both $n$ and $m$ are greater than $1$.

        (I) If $s\neq\frac{m+1}{2}$ and $r\neq\frac{n+1}{2}$, then
        $$\Aut(H^\ast(M_{\vs,\vr}))=\left\{\begin{pmatrix}1&0\\0&1\end{pmatrix}, \begin{pmatrix}-1&0\\0&-1\end{pmatrix}\right\}\cong \Z_2.$$
        Define a homeomorphism $f\colon S^{2n+1}\times S^{2m+1}\to S^{2n+1}\times S^{2m+1}$ by $$((w_1,\ldots,w_{n+1}), (z_1,\ldots,z_{m+1}))\mapsto ((\overline{w_1},\ldots,\overline{w_{n+1}}), (\overline{z_1},\ldots,\overline{z_{m+1}})).$$
        Then $f$ preserves the orbits of the action of $K_{\vs,\vr}$ defined in Section~\ref{sec:quasitoric manifolds with second betti 2}. Hence, $f$ induces a homeomorphism from $M_{\vs, \vr} = S^{2n+1} \times S^{2m+1} / K_{\vs,\vr}$ to itself.
        Let $\overline{f}$ be the homeomorphism induced from $f$. Then $\overline{f}^\ast$ is represented by the matrix $\begin{pmatrix}-1&0\\0&-1\end{pmatrix}$, and hence, $\{\overline{f}^\ast\}$ generates $\Aut(M_{\vs,\vr})$.

        (II) If $s=\frac{m+1}{2}$ and $r\neq\frac{n+1}{2}$, then
        \begin{equation*}
            \begin{split}
                \Aut(H^\ast(M_{\vs,\vr}))&=\left\{\begin{pmatrix}1&0\\0&1\end{pmatrix}, \begin{pmatrix}-1&0\\0&-1\end{pmatrix}, \begin{pmatrix}1&0\\-2&-1\end{pmatrix}, \begin{pmatrix}-1&0\\2&1\end{pmatrix}\right\}\\
                &\cong \Z_2\times \Z_2.
            \end{split}
        \end{equation*}
        Define a homeomorphism $g\colon S^{2n+1}\times S^{2m+1}\to S^{2n+1}\times S^{2m+1}$ defined by
        \begin{equation*}
        \begin{split}
            &((w_1,\ldots,w_{n+1}), (z_1,\ldots,z_{m+1}))\\
            &\qquad\mapsto ((w_1,\ldots,w_r,\overline{w_{r+1}},\ldots,\overline{w_{n+1}}), (z_{s+1},\ldots,z_{m+1},z_1,\ldots,z_s)),
        \end{split}
        \end{equation*}
        and then $g$ preserves the orbits of the action of $K_{\vs,\vr}$ on $S^{2n+1} \times S^{2m+1}$.

        Let $\overline{g}$ be the homeomorphism induced from $g$. Then $\overline{g}^\ast$ is represented by the matrix $\begin{pmatrix}-1&0\\2&1\end{pmatrix}$, and hence, the set $\{\overline{f}^\ast, \overline{g}^\ast\}$ generates $\Aut(H^\ast(M_{\vs,\vr}))$.

        (III) If $s\neq\frac{m+1}{2}$ and $r=\frac{n+1}{2}$, then
        \begin{equation*}
            \begin{split}
                \Aut(H^\ast(M_{\vs,\vr}))&=\left\{\begin{pmatrix}1&0\\0&1\end{pmatrix}, \begin{pmatrix}-1&0\\0&-1\end{pmatrix}, \begin{pmatrix}-1&-1\\0&1\end{pmatrix}, \begin{pmatrix}1&1\\0&-1\end{pmatrix}\right\}\\
                &\cong \Z_2\times \Z_2.
            \end{split}
        \end{equation*}
        Define a homeomorphism $h\colon S^{2n+1}\times S^{2m+1}\to S^{2n+1}\times S^{2m+1}$ defined by
        \begin{equation*}
        \begin{split}
            &((w_1,\ldots,w_{n+1}), (z_1,\ldots,z_{m+1}))\\
            &\qquad\mapsto ((w_{r+1},\ldots,w_{n+1},w_1,\ldots,w_r), (z_1,\ldots,z_s,\overline{z_{s+1}},\ldots,\overline{z_{m+1}})).
        \end{split}
        \end{equation*}
        and then $h$ also preserves the orbits of the action of $K_{\vs,\vr}$ on $S^{2n+1} \times S^{2m+1}$.

        Let $\overline{h}$ be the homeomorphism induced from $h$. Then $\overline{h}^\ast$ is represented by a matrix $\begin{pmatrix}1&1\\0&-1\end{pmatrix}$, and hence, the set $\{\overline{f}^\ast, \overline{h}^\ast\}$ generates $\Aut(H^\ast(M_{\vs,\vr}))$.

        (IV) If $s=\frac{m+1}{2}$ and $r=\frac{n+1}{2}$, then the set $\{\overline{f}^\ast, \overline{g}^\ast, \overline{h}^\ast\}$ generates \begin{equation*}
        \begin{split}
            &\Aut(H^\ast(M_{\vs,\vr}))\\
            &=\left\{\begin{pmatrix}1&0\\0&1\end{pmatrix}, \begin{pmatrix}-1&0\\0&-1\end{pmatrix}, \begin{pmatrix}1&0\\-2&-1\end{pmatrix}, \begin{pmatrix}-1&0\\2&1\end{pmatrix},\right.\\
            &\qquad\left.\begin{pmatrix}-1&-1\\0&1\end{pmatrix}, \begin{pmatrix}1&1\\0&-1\end{pmatrix}, \begin{pmatrix} 1&1\\-2&-1\end{pmatrix}, \begin{pmatrix}-1&-1\\2&1\end{pmatrix}\right\}.
        \end{split}
        \end{equation*}
    \end{proof}

\section{Proofs of Theorems~\ref{thm:submain} and~\ref{thm:main}}\label{sec:proof of main}

    Let $M$ and $M'$ be quasitoric manifolds with second Betti number $2$. If $\varphi\colon H^\ast(M)\to H^\ast(M')$ is an isomorphism as graded rings, then there exists a homeomorphism $f\colon M\to M'$ by Theorem~\ref{thm:main of pjm}. Then $f$ induces a graded ring isomorphism $f^\ast\colon H^\ast(M')\to H^\ast(M)$. Accordingly, $f^\ast\circ\varphi$ is a ring automorphism of $H^\ast(M)$. Hence, by Proposition~\ref{prop:two-stage GB} and Proposition~\ref{prop:NGB}, there exists a homeomorphism $g\colon M\to M$ such that $g^\ast=f^\ast\circ\varphi$. Hence, $\varphi$ is realizable by a homeomorphism $g\circ f^{-1}$, as shown in the following diagram. This proves Theorem~\ref{thm:main}.
    \begin{equation*}
        \xymatrix@1{
         H^\ast(M)\ar[dd]^{\varphi}\ar[dr]^{g^\ast}&\\
         & H^\ast(M)\\
         H^\ast(M')\ar[ur]_{f^\ast}&
        }
    \end{equation*}

     Furthermore, if $M$ and $M'$ are non-singular complete toric varieties, then $f$ and $g$ are diffeomorphisms by Theorem~\ref{thm:CR of GB} and Proposition~\ref{prop:two-stage GB}. This proves Theorem~\ref{thm:submain}.

\end{document}